\documentclass[11pt]{article}
\usepackage{amsmath,amsthm,amsfonts,amssymb,amscd, amsxtra,color}
\usepackage[active]{srcltx}
\usepackage{url}
\usepackage{cite}
 \usepackage{graphicx}
\usepackage{multirow}
\usepackage[margin=2.5cm,nohead]{geometry}
\theoremstyle{plain}
\theoremstyle{definition}
\newtheorem{theorem}{Theorem}
\newtheorem{lemma}{Lemma}
\newtheorem{definition}{Definition}
\newtheorem{corollary}{Corollary}
\newtheorem{proposition}{Proposition}

\newtheorem{remark}{Remark}
\newtheorem{example}{Example}
\newcommand{\beq}{\begin{equation}}
\newcommand{\eeq}{\end{equation}}
\newcommand{\tp}{^\top}
\newcommand{\p}{\partial}
\newcommand{\lng}{\langle}
\newcommand{\rng}{\rangle}
\newcommand{\lf}{\left}
\newcommand{\rg}{\right}
\newcommand{\f}{\frac}

\def\argmin{\operatorname{argmin}}

\def\min{\operatorname{min}}
\def\max{\operatorname{max}}

\DeclareMathOperator{\diag}{diag}
\DeclareMathOperator{\inte}{int}

\DeclareMathOperator{\dist}{d}
\DeclareMathOperator{\sgn}{sgn}

\newcommand{\pr}{\perp}
\newcommand{\mc}{\mathcal}

\newcommand{\SP}{\mathbb S}
\newcommand{\R}{\mathbb R}

\begin{document}
\title{On the spherical  quasi-convexity of quadratic functions \\ on   spherical self-dual  convex sets}

\author{O. P. Ferreira\thanks{IME/UFG, Avenida Esperan\c{c}a, s/n, Campus II,  Goi\^ania, GO - 74690-900, Brazil (E-mails: {\tt orizon@ufg.br}).}
\and
S. Z. N\'emeth \thanks{School of Mathematics, University of Birmingham, Watson Building, Edgbaston, Birmingham - B15 2TT, United Kingdom
(E-mail: {\tt s.nemeth@bham.ac.uk}, {\tt LXX490@student.bham.ac.uk}).}
\and
L. Xiao  \footnotemark[3]
}
\maketitle
\begin{abstract}
In this paper,  the spherical quasi-convexity of quadratic functions on spherically subdual convex sets is
studied. Sufficient conditions for spherical quasi-convexity on spherically subdual convex sets are presented. A partial characterization of spherical quasi-convexity on spherical Lorentz sets is given and some examples are provided.
\end{abstract}

\section{Introduction}
The aim of this paper is to  study theoretical properties of spherical quasi-convexity of quadratic functions on spherically subdual convex sets. It is well known that quadratic functions  play an important role in nonlinear programming theory.   For instance,  the minimization problem of
 quadratic functions on the sphere occurs as a subproblem in methods of nonlinear programming  (see the background  section of
\cite{RendlWolkowicz1997}  for an extensive review of the literature on the subject).     We are interested in the problem  of minimizing a quadratic function, defined by the symmetric matrix $ Q $, constrained to a subset $ C $ of the sphere. This problem is a quadratic constrained optimization problem on the sphere, and it is also a minimum eigenvalue problem in $C$. It is essential to emphasize that there exists a special case,  when $C$ is the intersection of the Lorentz cone with the sphere. This special case is of particular interest because  the minimum eigenvalue of $Q$ in $C$ is non-negative,  if and only if the matrix $Q$ is Lorentz copositive; see  \cite{LoewySchneider1975,GajardoSeeger2013}.  In general, changing the Lorentz cone by an
arbitrary closed and convex cone $K$ would lead to a more general concept of $K$-copositivity, thus our study is anticipated to initialize new perspectives for investigating the general copositivity of a symmetric matrix. In general, exploiting the specific intrinsic geometric and algebraic structure of problems posed on the sphere can significantly lower down the cost of finding solutions; see \cite{Hager2001, HagerPark2005, Smith1994,So2011,Zhang2003,Zhang2012}. We know that  a strict local minimizer of a spherically quasi-convex quadratic function is also a strict global minimizer, which makes interesting and natural to refer the problem about characterizing the spherically quasi-convex quadratic functions on spherically convex sets.

	The aim of this paper is to introduce both sufficient conditions and necessary conditions for quadratic functions to be spherically quasi-convex on  spherically subdual convex sets. In particular, several examples are presented.  This paper continues the study of
\cite{FerreiraNemethXiao2018}, which can be regarded as premier study about the topic of quasi-convexity of quadratic functions on
the Euclidean space. The main literature about the quasi-convexity of quadratic functions on Euclidean convex sets includes, but it is not
limited to  \cite{Martos1969,Ferland1972,Schaible1981,Komlosi1984,Karamardian1993}.

	The remaining part of this paper is structured as follows. Section~\ref{sec:int.1} presents previous results and notations,  that will be used throughout this paper. In Section~\ref{sec:qcqfcs},  we recall the fundamental properties of spherically quasi-convex functions on spherically convex sets and,  in Section~\ref{sec:sqcqfcs},   particular versions of these conditions for  quadratic spherically quasi-convex functions.
Section~\ref{sec:qcqfcsd} provides derivations of many useful properties of spherically quasi-convex functions on spherically subdual convex sets.
In Section~\ref{sec:qcqflc},  we prove a condition partially characterizing the  spherical quasi-convexity of
quadratic functions on  spherically convex sets associated  with  the Lorentz  cone.  Perspectives and open problems are presented in Section~\ref{sec:pop} and Section~\ref{sec:Conclusions}  concludes  the paper.

\section{Terminology and Basics Results} \label{sec:int.1}
In this section, we introduce some notations and present the previous results used throughout the paper. Let $\R^n$ be the $n$-dimensional Euclidean
space with the canonical inner product $\langle\cdot,\cdot\rangle$ and  norm $\|\cdot\|$.  For  a set $D\subset \R^n$, denote by  $[D]$  the  {\it
span} of a set $D$, i.e.,   the smallest linear subspace of $\R^n$  that contains the set $D$. For a vector subspace  $\mathbb{V} \subset \R^n$,
denote by $ \mathbb{V}^\pr=\{x\in\R^n:~\lng v,x\rng=0, ~\forall ~ v \in \mathbb{V}\}$ its orthogonal complement in $\R^n$. In particular, for
simplifying the notation, for a vector $u\in\R^n$ we set $[u]:=[\{u\}]$ and $[u]^\pr:=[\{u\}]^\pr$. Denote by $\R^{m \times n}$ the set of all $m\times n$
matrices with real entries, $\R^n\equiv \R^{n \times 1}$,   by $e^i$ the $i$-th canonical unit vector in $\R^n$, and by ${\rm I_n}$  the $n\times
n$ identity matrix.  A set ${\cal K}\subseteq\R^{n}$ is called a {\it cone} if for any $\alpha> 0$ and $x\in \cal{K}$ we have $\alpha x\in \cal{K}$. A cone
${\cal K}\subseteq\R^n$ is called a {\it convex cone} if for any $x,y \in\cal{K}$, we have $x + y \in \cal{K}$. The {\it dual cone} of a cone ${\cal{K}} \subseteq
\R^n$ is the closed convex cone ${\cal{K}}^*\!\!:=\!\{ x\in \R^n :~ \langle x, y \rangle\!\geq\! 0, ~ \forall \, y\!\in\! {\cal{K}}\}.$  A cone $ {\cal K}
\subseteq \R^{n}$  is called {\it pointed} if $ {\cal K} \cap {(-\cal K)} \subseteq \{0\}$. A pointed closed convex cone is called \emph{proper
cone} if it has a nonempty interior. The cone ${\cal K}\subseteq\R^n$ is called {\it subdual}
if ${\cal K}\subseteq {\cal K}^*$ and {\it self-dual} if ${\cal K}={\cal K}^*$. The {\it Lorentz cone}  is defined by
\begin{equation*} \label{eq:LorentzCone}
	{\cal L}:=\left\{x=(x_1,\dots,x_n)^\top\in\R^n:~x_1\geq\sqrt{(x_2)^2+\dots+(x_n)^2}\right\}, 
\end{equation*}
which  is a special case of the elliptic cone defined by
$$
	{\cal E}:=\left\{x\in\R^n:~\lng v^1,x\rng\geq  \sqrt{\theta_2\lng v^2,x\rng^2+\dots+\theta_{n}\lng
	v^n,x \rng^2}\right\}, 
$$
where $ \theta_2, \ldots,\theta_n$ are  positive real  numbers  and $v^1, v^2, \ldots, v^n$ are L.I. vectors in $\R^n$.  We recall that the Lorentz cone ${\cal L}$ and the non-negative orthant   $\R^n_{+}$ are  self-dual cones.
Let ${\cal K}$ be a closed convex cone.    Let $x\in \R^n$, then  the {\it projection $P_{\cal K}(x)$ of the point $x$ onto the cone  \( {\cal K}\)} is defined by 
$$
P_{\cal K}(x):=\argmin \left\{\|x-y\|~:~ y\in {\cal K}\right\}.
$$
For any $x\in {\cal K}$, we define the non-negative part of $x$, nonpositive part of $x$ and the absolute value of $x$ with respect to ${\cal K}$ by
\begin{equation} \label{eq;npav}
	x_+^{\cal K}:={\rm P}_{\cal K}(x), \qquad x_-^{\cal K}:={\rm P}_{{\cal K}^*}(-x), \qquad  |x|^{\cal K}:=x_+^{\cal K}+x_-^{\cal K},
\end{equation}
respectively. We recall   from   Moreau's  decomposition theorem \cite{Moreau1962} (see also  \cite[Theorem 3.2.5]{HiriartLemarecal1}),   that
for a closed convex cone ${\cal K}$  there hold:
\begin{equation} \label{eq;cmdt}
	x=x_+^{\cal K}-x_-^{\cal K}, \qquad \left\langle x_+^{\cal K}, x_-^{\cal K} \right\rangle=0,   \qquad \qquad  ~x\in\R^n.
\end{equation}
For any $z\in \mathbb{R}\times
{\mathbb{R}}^{n-1}$, let
\(z:=(z_1, {z^2})  \in \mathbb{R}\times {\mathbb{R}}^{n-1} \), where ${z^2}:= (z_2, z_3, \dots, z_n)^{\tp}$.
An explicit formula for the projection mapping    \({\rm P}_{\cal{L}}\)  onto the  Lorentz cone  ${\cal L}$ is given in  \cite[Proposition
3.3]{FukushimaTseng2002}, which is recalled for the case when $x\notin{\cal L}\cup -{\cal L}$ in the following lemma.
\begin{lemma} \label{l:projude-}
Let \(x=(x_1, {x^2})  \in  \{(y_1, {y^2}) \in \R \times \R^{n-1}: ~ |y_1|<\|{y^2}\|\}  \) and \({\cal{L}}\) be the Lorentz cone. Then,
$$
x_+^{\cal L}= \lf(\frac{x_1+\|{x^2}\|}{2\|x^2\|}\rg)\left(\|x^2\|,x^2\right), \qquad
x_-^{\cal L}= \lf(\frac{-x_1+\|{x^2}\|}{2\|x^2\|}\rg)\left(\|x^2\|,-x^2\right)
$$
and, as a  consequence, the absolute value of $x$ with respect to ${\cal L}$ is given by
\[|x|^{\cal L} = \frac1{\|x^2\|}\left(\|{x^2}\|^2,x_1x^2\right).\]
\end{lemma}
For a general nonzero vector \(x=\lf(x_1, {x^2}\rg)  \in  \R \times \R^{n-1}\) the absolute value of $x$ with respect to ${\cal L}$ is given in
the next
lemma, which follows immediately from Lemma \ref{l:projude-} and equations \eqref{eq;cmdt}.
\begin{lemma} \label{l:projude}
Consider a nonzero vector \(x=\lf(x_1, {x^2}\rg)  \in  \R \times \R^{n-1}\) and let \({\cal{L}}\) be the Lorentz cone. Then, the absolute value of $x$  is given by
\[|x|^{\cal L} = \frac1{\|x^2\|}\Big{(}\max\left(|x_1|,\|{x^2}\|\right)\|x^2\|,~\min(|x_1|,\|x^2\|)\sgn(x_1)x^2\Big{)},\]
where  $\sgn(x_1)$ is equal to $-1$, $0$ or $1$ whenever  $x_1$ is negative, zero or positive, respectively.  
\end{lemma}
Let \( {\cal K} \subseteq \R^{n}\)  be a (not necessarily convex) cone.  Let us recall that   $A\in \R^{n \times n}$ is  ${\cal K}$-{\it copositive}
if  $ \langle Ax, x \rangle\!\geq\! 0$ for all $x\in {\cal K}$  and  a {\it Z-matrix} is a matrix with nonpositive off-diagonal elements.
It is easy to see that the Lorentz cone \({\cal{L}}\) can be written as
\begin{equation*} \label{eq:LorentzConeJ}
	{\cal L}:=\left\{x=(x_1,\dots,x_n)^\top\in\R^n:~x_1\geq 0,\textrm{ }\lng Jx,x\rng\ge0\right\},
\end{equation*}
where $J=\diag(1,-1,\dots,-1)\in\R^{n\times n}$.
It is easy to see that \[{\cal L}\cup-{\cal L}=\left\{x=(x_1,\dots,x_n)^\top\in\R^n:\lng Jx,x\rng\ge0\right\}.\]
	This straightforwardly implies that  $A\in \R^{n \times n}$ is ${\cal L}$-copositive if and only if it is ${\cal L}\cup-{\cal L}$-copositive.
Hence, the S-Lemma (see \cite{Jabukovich1971,Jabukovich1977,PolikTerlaky2007}) implies:
\begin{lemma}\label{lorcop}
	$A\in \R^{n \times n}$ is ${\cal L}$-copositive if and only if there exists a $\rho\ge0$ such that $A-\rho J$ is positive
	semidefinite.
\end{lemma}
	This result is well-known and it is remarked in Lemma 2.2 of \cite{LoewySchneider1975}.
Let ${\cal K} \subseteq \R^n $ be  a  pointed closed convex cone  with nonempty interior,  the  {\it ${\cal K}$-Z-property} of a matrix $A\in\R^{n\times n}$ means that
$ \langle Ax,y\rangle \le0$, for any $(x,y)\in C({\cal K})$, where $C({\cal K}):=\{(x,y)\in\R^n\times\R^n:~x\in {\cal K},\textrm{ }y\in {\cal
K}^*,  \langle x, y \rangle=0\}$.  Throughout the paper the {\it \(n\)-dimensional  Euclidean sphere}  $\SP^{n-1}$ is  denoted by
\[
\SP^{n-1}:=\left\{ x=(x_1, \ldots, x_{n})\tp \in \R^{n}~:~ \|x\|=1\right\}.
\]
The  {\it intrinsic distance on the sphere} between  \(x, y \in \SP^{n-1}\)  is  defined by $\dist(x, y):=\arccos \langle x , y\rangle.$  It is also easy to verify that $\dist(x, y)\leq \pi$ for any $x, y\in \SP^{n-1}$,
and $\dist(x, y)=\pi$ if and only if $x=-y$. The intersection curve of a plane through the origin of \(\R^{n}\)
with the sphere \( \SP^{n-1}\) is called a { \it geodesic}.   A geodesic segment joining $x$ to $y$  is said to be  {\it   minimal } if its length is equal to $\dist(x, y)$.
A set \({\cal C} \subseteq \SP^{n-1}\) is said to be  \emph{spherically convex} if for any \(x\), \(y\in {\cal C} \)
all the minimal geodesic segments joining \(x\) to \(y\)  are contained in \( {\cal C} \).
For notational convenience, in the following text we assume that {\it all spherically convex sets are nonempty proper
subsets of the sphere}.  For each closed set \( {\cal A} \subseteq \SP^{n-1}\), let \({\cal K}_{\cal A}\subseteq\R^{n}\) be the {\it cone spanned by} \({\cal A}\),
namely,
\begin{equation*} \label{eq:pccone}
{\cal K}_{\cal A}\:=\left\{ tx \, :\, x\in {\cal A}, \; t\in [0, +\infty) \right\}.
\end{equation*}
\begin{figure}[h!]
	\centering
	\includegraphics[width=.6\textwidth, height=.5\textwidth]{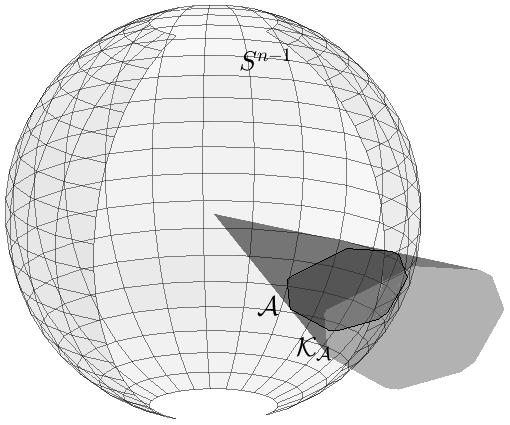} \\
	 \caption{ Closed set \({\cal A}\) and the cone \({\cal K}_{\cal A} \) spanned by \({\cal A}\).~~~~~~~~~} 
	\label{fig:soc}
\end{figure}
Obviously, \({\cal K}_{\cal A} \) is the smallest cone containing \({\cal A}\) (see Fig.
\ref{fig:soc}).  In the following proposition, a relationship between spherically convex sets and the cones spanned by them will be exhibited. The proof is  given  in  \cite[Proposition 2]{FerreiraIusemNemeth2013}.
\begin{proposition} \label{pr:ccs}
	The set \({\cal C}\) is spherically convex if and only if the cone \(\mathcal{K}_{\cal
	C}\) is pointed convex.
\end{proposition}
\subsection{Spherically quasi-convex  functions  on spherically convex sets} \label{sec:qcqfcs}
In this section we recall the concept of spherically quasi-convex  functions  on spherically convex sets and we present a basic characterization
of them; for more details see  \cite{FerreiraNemethXiao2018}.
\begin{definition}\label{def:qcf-b}
Let \( {\cal C}\subseteq\SP^{n-1}\) be a spherically convex set and $I\subseteq \R$ be  an interval.
A function \(f:{\cal C} \to \R\)  is said to be spherically quasi-convex (respectively, strictly spherically quasi-convex)
if for any minimal geodesic  segment \(\gamma:I\to {\cal C}\), the composition
\( f\circ \gamma :I\to \R\) is quasi-convex (respectively, strictly quasi-convex) in the usual sense, i.e.,
\(f(\gamma(t))\leq \max \{ f(\gamma(t_{1})), f(\gamma(t_{2}))\}\) for all \(t\in [t_{1}, t_{2}]\subseteq I\), (respectively,   \(f(\gamma(t))< \max
\{ f(\gamma(t_{1})), f(\gamma(t_{2}))\}\)  for all \(t\in [t_{1}, t_{2}]\subseteq I\), $t_1\ne t_2$).
\end{definition}

\begin{remark}
The above definition implies that,  if \(f: \SP^{n-1} \to \R\)  is  spherically quasi-convex, then $f$ is constant. However, as wee will show,
there exist  non-constant spherically quasi-convex functions  defined on  proper convex sets  of $\SP^{n-1}$.
\end{remark}
To simplify the notations, the {\it sub-level sets} of a function  \(f:\R^n\supseteq{\cal M} \to \R\)   are denoted by
\begin{equation*} \label{eq:sls}
[f\leq c]:=\{x\in {\cal M} :\; f(x)\leq c\}, \qquad c\in \R.
\end{equation*}
\begin{proposition} \label{pr:charb1}
Let \( {\cal C}\subseteq\SP^{n-1}\) be a spherically convex set. A function \(f:{\cal C} \to \R\) is spherically quasi-convex if and only if for all \(c\in \R\) the sub-level sets \([f\leq c]\) are spherically convex.
\end{proposition}
\subsubsection{Spherically quasi-convex quadratic functions  on spherically convex sets} \label{sec:sqcqfcs}
In this section we recall earlier results of  quadratic  quasi-convex functions on  general  spherically convex sets.   {\it Henceforward we assume that the cone ${\cal K}\subseteq \R^{n}$ is proper and subdual, the set ${\cal C}=\SP^{n-1}\cap\inte({\cal K})$ is open and spherically convex, and  $A=A^T\in\R^{n\times n}$}.   Let $q_A: {\cal C}\to\R$ be   the quadratic function defined by
\begin{equation} \label{eq:QuadFunc}
q_A(x):=\langle Ax,x\rangle.
\end{equation}
To proceed we need as well the  {\it restriction  on $\inte {\cal K} $ of  Rayleigh quotient }  $\varphi_A :  \inte {\cal K} \to \R$  defined by
\begin{equation} \label{eq:RayleighFunction}
\varphi_A(x):=\frac{\langle Ax, x \rangle}{\|x\|^2}.
\end{equation}
In the followings  we state  some   properties  of the functions   $q_A$  and $\varphi_A(x)$, for details see \cite{FerreiraNemethXiao2018}.
\begin{proposition}\label{pr:spher-quasiconv}
 The following statements  are equivalent:
	\begin{enumerate}
		\item[(a)]   $q_A$ is spherically quasi-convex;
		\item [(b)] $\langle Ax,y\rangle\leq\langle x,y\rangle\max \left\{q_A(x), ~q_A(y)\right\}$,  for all $ x,y\in\SP^{n-1}\cap {\cal K}$;
		\item [(c)] $ \displaystyle \frac{\langle Ax,y\rangle}{\langle x,y\rangle}\leq\max\left\{\varphi_A(x), ~\varphi_A(y)\right\}$,  for all $x,y\in {\cal K}$ with $\langle x,y\rangle\ne 0$.
	\end{enumerate}
\end{proposition}
\begin{corollary}\label{cor:qpcf}
Let  ${\cal K}$ be   self-dual.  If  $q_A$ is spherically quasi-convex, then $A$ has the $\cal K$-Z-property.
\end{corollary}
\begin{theorem}\label{th:quasiconv-iff}
 The function $q_A$   is spherically quasi-convex if and only if  $\varphi_A$   is quasi-convex.
\end{theorem}
The next result uses the following notations:  Let $c\in \R$ and define the cone
\begin{equation*}\label{eq:closure}
[\varphi_A\leq c]:=\{x\in{\cal K}:\langle A_{c}x,x\rangle\leq 0\}, \qquad   A_{c}:=A-c{\rm I_n}.
\end{equation*}
\begin{corollary}\label{cor:cor}
 The function $q_A$   is spherically quasi-convex  if and only if $[\varphi_A\leq c]$ is convex,  for any $c\in\R$.
\end{corollary}
\section{Spherically Quasi-Convex Quadratic Functions  on Spherically Subdual Convex Sets} \label{sec:qcqfcsd}
In this section we present    partial conditions characterizing  the  spherical   quasi-convexity of quadratic
functions on  spherically subdual convex
sets associated to subdual  cones.  The results obtained  generalize the corresponding ones   obtained in \cite[Section
4.1]{FerreiraNemethXiao2018}. In due course we will present a more precise correspondence.   Throughout
this section we assume that ${\cal K}$ is  subdual, i.e., ${\cal K}\subseteq{\cal K}^*$ and proper.  A  closed set \( {\cal
A} \subseteq \SP^{n-1}\) is called  {\it spherically subdual convex set} if the associated cone ${\cal K}_{\cal A}$ is subdual. It is clear that
if $A=A\tp\in\R^{n\times n}$ has only one eigenvalue, then  $q_A$ is constant and, consequently,    it is   spherically quasi-convex. Henceforth,
throughout  this section {\it we assume that $A$  has at least two distinct  eigenvalues}.   Let us recall that   $q_A$ and $\varphi_A$  are defined in
\eqref{eq:QuadFunc} and \eqref{eq:RayleighFunction}, respectively. Two technical lemmas, which are
useful in the following text, will be presented. They are generalizations of
Lemmas~14~and~15~of~\cite{FerreiraNemethXiao2018}, respectively.   More specifically,  in Lemma~14
of \cite{FerreiraNemethXiao2018}  a condition implying  the convexity of  $[\varphi_A\leq c]$,
for all  $c\notin(\lambda_2,\lambda_n)$,   is presented if $K={\mathbb R}^n_{+}$,  while here in
Lemma~\ref{Lem:Basic} that condition is extended to an arbitrary subdual cone.   Similarly,
Lemma~\ref{lem:Copositive} generalizes Lemma~15~of~\cite{FerreiraNemethXiao2018} from $K={\mathbb
R}^n_{+}$ to an arbitrary subdual cone.  For stating the next lemma, for $\{v^1, v^2,\dots,v^n\}$ a orthonormal system of eigenvectors of $A$  corresponding to the
eigenvalues $\lambda_1 < \lambda_2 \leq  \ldots \leq \lambda_n$, respectively, and  $c\in(\lambda_1,\lambda_2]$, define  the  convex   cone 
\begin{equation}\label{eq-ltheta}
	{\cal L}_c:=\left\{x\in\R^n:~\lng v^1,x\rng\geq  \sqrt{\theta_2(c)\lng v^2,x\rng^2+\dots+\theta_{n}(c)\lng
	v^n,x \rng^2}\right\},  \qquad \theta_i(c):=\frac{\lambda_i-c}{c-\lambda_1}, 
\end{equation}
for $i=2, \ldots, n$.  Note that if $\lambda_1<c<\lambda_2$, then $\theta_i(c)> 0$, for $i=2, \ldots, n$, and $\mathcal L_{c}$, $-\mathcal L_{c}$ are also a pointed proper  elliptic cones.   We also need to consider  the following cone
\begin{equation} \label{eq:cw}
{\mathcal W}:=(\mathcal L_{\lambda_2}\cup-\mathcal L_{\lambda_2})\cap\inte({\mathcal K}).
\end{equation} 
Considering that ${\mathcal K}$ is a proper cone,  then the cone ${\mathcal W}$ is  also  proper, i.e., $\inte({\mathcal W})\neq \varnothing$.  
\begin{lemma}\label{Lem:Basic}
Let $n\geq 2$,     $A=A\tp\in\R^{n\times n}$ and   $\{v^1, v^2,\dots,v^n\}$ be an orthonormal system of eigenvectors of $A$  corresponding to the
eigenvalues $\lambda_1 < \lambda_2 \leq  \ldots \leq \lambda_n$, respectively.  Then, the sublevel set $[\varphi_A\leq c]$ is convex for all $c\notin(\lambda_2,\lambda_n)$ if and only if
$v^1 \in {\cal W}^*\cup-{\cal W}^*$. In particular if $v^1 \in {\cal K}^*$, then $[\varphi_A\leq c]$ is convex for all 
$c\notin(\lambda_2,\lambda_n)$.
\end{lemma}
\begin{proof} By using the spectral decomposition of $A$,
			we have  $A=\sum_{i=1}^n\lambda_iv^i(v^i)\tp$. From  \eqref{eq:RayleighFunction} we have 
			\begin{equation}\label{eq:bl1} 
					[\varphi_A\leq c]=\Big\{x\in\inte ({\cal K}):~ \sum_{i=1}^n(\lambda_i-c)\langle v^i, x\rangle^2  \leq
					0\Big\}.
			\end{equation}
			If $\lambda_1<c\le\lambda_2$, then by using \eqref{eq-ltheta} the equality  \eqref{eq:bl1} can be completed as follows
			\begin{align}\label{eq:bl1-2} 
					{\mathcal W}=[\varphi_A\leq\lambda_2]\supseteq[\varphi_A\leq c]&=({\mathcal L}_c\cup-{\mathcal
					L}_c)\cap\inte({\mathcal K}) \notag \\
					                                                                                                        &=\left\{x\in\inte(\mathcal K):~\lng v^1,x\rng^2\geq\theta_2(c)\lng v^2,x\rng^2+\dots+\theta_{n}(c)\lng v^n,x \rng^2\right\}.
			\end{align}
\noindent			
 \textbf{\it Sufficiency of the first statement:} 
			Let $v^1\in {\cal W}^*$ (a similar argument holds for $v^1\in-{\cal W}^*$).  If $c<\lambda_1$, then considering that  $v^1,
			v^2, \ldots , v^n$ are linearly independent and $0\notin \inte ({\cal K})$, we obtain from \eqref{eq:bl1} that
			$[\varphi_A\leq c]=\varnothing $ and hence it is convex.  If $c=\lambda_1$, then  \eqref{eq:bl1} implies that
			$[\varphi_A\leq c]={\cal S}\cap \inte ({\cal K}),$ where  ${\cal S}: =\{x\in\R^n ~: \langle v^i, x\rangle=0, ~\mbox{for}~ i=2, \ldots, n\}$.
			Thus, due to $\inte ({\cal K})$ and ${\cal S}$ being convex, we conclude that $[\varphi_A\leq c]$ is also convex. Now, we
			suppose  that $\lambda_1<c\le\lambda_2$. Since $v^1\in {\cal W}^*$, for any $x\in\mathcal W$ we obtain that $\lng
			v^1,x\rng\ge0$ and from \eqref{eq:bl1-2} we have  $[\varphi_A\leq c]= {\cal L}_c\cap \inte ({\cal K})$. Due to the 
			convexity
			of the cones   ${\cal L}_c$ and $\inte ({\cal K})$, we obtain  that  $[\varphi_A\leq c]$ is convex.   Finally, if $c\ge
			\lambda_n$, then \eqref{eq:bl1} implies that $[\varphi_A\leq c]=\inte ({\cal K})$ is convex.\\
\noindent	
\textbf{\it Necessity of the first statement:}  We will show that $v^1\notin {\cal W}^*\cup-{\cal W}^*$ implies that $[\varphi_A\leq c]$ is not convex, for some $c\in (\lambda_1,\lambda_2)$.
			Suppose that $v^1\notin {\cal W}^*\cup-{\cal W}^*$.  Thus, considering that $\inte(\mathcal W)\neq \varnothing$, there exist $y,z\in\inte(\mathcal W)$ such that 
			$\lng v^1,y\rng>0$ and $\lng v^1,z\rng<0$. Thus, \eqref{eq-ltheta} and \eqref{eq:cw}  imply  that  
			\begin{equation} \label{eq:yz}
			y\in\inte(\mathcal K)\cap\inte(\mathcal L_{\lambda_2}), \qquad \quad   z\in\inte(\mathcal K)\cap\inte(-\mathcal L_{\lambda_2}).
			\end{equation}
			We claim that there exists ${\bar c}\in (\lambda_1,\lambda_2)$ such that $y\in\inte(\mathcal K)\cap \inte({\mathcal L}_{{\bar c}})$ and  $z\in\inte(\mathcal K)\cap \inte(-\mathcal L_{{\bar c}})$.  In order to simplify the notations, for  $ x\in\R^n$ and   $c\in (\lambda_1,\lambda_2]$, we define the following function   
			\begin{equation} \label{eq:fpsy}
			\psi(x,c):=\sqrt{\theta_2(c)\lng v^2,x\rng^2+\dots+\theta_{n}(c)\lng v^n,x \rng^2}.
			\end{equation}
			Note that $\psi$ is  a  continuous function  and,  from   the definition of $\theta_i$ in   \eqref{eq-ltheta},  it is also
			decreasing with respect to the second variable $c$. By using definitions \eqref{eq-ltheta} and \eqref{eq:fpsy} we have 
			\begin{equation} \label{eq:edef}
			\inte(\mathcal K)\cap \inte({\cal L}_c)=\left\{x\in\inte{\mathcal K}:~\lng v^1,x\rng>\psi(x,c)\right\}, \qquad \forall ~c\in (\lambda_1,\lambda_2].
			\end{equation}		
			 Thus,  taking into account  the  first inclusion in \eqref{eq:yz}  we conclude,  by setting   $c=\lambda_2$ in   \eqref{eq:edef},    that 
			\[
			\lim_{c\to\lambda_2}\psi(y,c)=\psi(y,\lambda_2)<\lng v^1,y\rng. 
			\] 
			Hence,  there exists a ${\hat c}\in(\lambda_1,\lambda_2)$ sufficiently close to $\lambda_2$ such that $\psi(y,{\hat c})< \lng v^1,y\rng$. Similarly, we can also prove that   there exists  a ${\tilde c}\in(\lambda_1,\lambda_2)$ sufficiently close to $\lambda_2$ such that $\psi(z,{\tilde c})< -\lng v^1,z\rng$. Thus, letting ${\bar c}=\max\{{\hat c}, {\tilde c} \}$ we conclude that $\psi(y,{\bar c})< \lng v^1,y\rng$ and  $\psi(z,{\bar c})< -\lng v^1,z\rng$, which by   \eqref{eq:fpsy} and \eqref{eq:edef} yields
			\begin{equation} \label{eq;icba}
			y\in \inte({\mathcal L}_{{\bar c}}), \qquad \quad z\in\inte(-\mathcal L_{{\bar c}}). 
			\end{equation} 
We know by 	 \eqref{eq:yz} that 	$y\in\inte(\mathcal K)$ and  $z\in\inte(\mathcal K)$, which together  with \eqref{eq;icba} yields
$y\in\inte(\mathcal K)\cap \inte({\mathcal L}_{{\bar c}})$ and  $z\in\inte(\mathcal K)\cap \inte(-\mathcal L_{{\bar c}})$ and the  claim is
concluded.  Therefore,  there exist  $r_y>0$ and  $r_z>0$ such $B(y,r_y)\subset \inte(\mathcal K)\cap \inte({\mathcal L}_{{\bar c}})$ and
$B(z,r_z)\subset \inte(\mathcal K)\cap \inte(-\mathcal L_{{\bar c}})$, where   $B(y,r_y)$ and  $B(z,r_z)$ denote  the open balls  with  centres
$y$, $z$ and  radii $r_y>0$, $r_z>0$. Hence, by  dimensionality reasons, we can take  $u_y\in \inte(\mathcal K)\cap \inte({\mathcal L}_{{\bar
c}})$ and  $u_z \in \inte(\mathcal K)\cap \inte(-\mathcal L_{{\bar c}})$ such that $v^1$,  $u_y$ and $u_z$ are linearly independent (L.I.).  Thus,
in particular, we have  $0\notin [u_y,u_z]$,  where  $[u_y,u_z]$ denotes the straight line segment joining $u_y$ to $u_z$. Since $ \inte({\mathcal
L}_{{\bar c}})\cap \inte(-\mathcal L_{{\bar c}})=~\varnothing$ and $0\notin [u_y,u_z]$, the  segment $[u_y,u_z]$ is intersecting, at the distinct
points
$w_y\neq 0$ and $w_z\neq 0$,   the boundaries of the sets  $ \inte({\mathcal L}_{{\bar c}})$ and $  \inte(-\mathcal L_{{\bar c}})$, respectively.
Moreover, due to   $u_y$ and $u_z$  being  L.I.,  $0\notin [u_y,u_z]$,  and    $w_y, w_z\in  [u_y,u_z]$, we conclude that the vectors $v^1$,  $w_y$
and $w_z$ are  also L. I. Our next task is to  prove  that  $(w_y+w_z)/2$ does not belongs to $ {\mathcal L_{\bar c}}\cup-{\mathcal L_{\bar c}}$, i.e., 
\begin{equation} \label{eq:mwywz}
\frac{1}{2}(w_y+w_z)\notin  {\mathcal L_{\bar c}}\cup-{\mathcal L_{\bar c}}.
\end{equation}
First, due to $w_y$ and $ w_z$ belonging  to  the boundaries  of  $ {\mathcal L}_{{\bar c}}$  and $  -\mathcal L_{{\bar c}}$, respectively,  we obtain from \eqref{eq-ltheta}  that
\begin{equation} \label{eq:ewywz}
\lng v^1,w_y\rng=  \sqrt{\sum_{i=2}^n\theta_i({\bar c})\lng v^i,w_y\rng^2}, \qquad \lng v^1,w_z\rng= - \sqrt{\sum_{i=2}^n\theta_i({\bar c})\lng v^i,w_z\rng^2}.
\end{equation}
On the other hand, by using the two equalities in  \eqref{eq:ewywz},  we obtain after some algebraic manipulations that 
$$
\sum_{i=2}^n\theta_i({\bar c})\left\lng v^i,\frac{1}{2}(w_y+w_z) \right\rng^2= \left\lng v^1,\frac{1}{2}w_y\right\rng^2+ \left\lng v^1,\frac{1}{2}w_z \right\rng^2+ 2\sum_{i=2}^n\theta_i({\bar c})\left\lng v^i,\frac{1}{2}w_y\right\rng\left\lng v^i,\frac{1}{2}w_z\right\rng.
$$
Thus, considering that   $\left\lng v^1,\frac{1}{2}(w_y+w_z) \right\rng^2= \left\lng v^1,\frac{1}{2}w_y\right\rng^2+ \left\lng v^1,\frac{1}{2}w_z \right\rng^2+ 2 \left\lng v^1,\frac{1}{2}w_y\right\rng\left\lng v^1,\frac{1}{2}w_z \right\rng,$ we have 
\begin{multline} \label{eq:fil3}
\sum_{i=2}^n\theta_i({\bar c})\left\lng v^i,\frac{1}{2}(w_y+w_z) \right\rng^2=\left\lng v^1,\frac{1}{2}(w_y+w_z) \right\rng^2- 2 \left\lng v^1,\frac{1}{2}w_y\right\rng\left\lng v^1,\frac{1}{2}w_z \right\rng \\+ 2\sum_{i=2}^n\theta_i({\bar c})\left\lng v^i,\frac{1}{2}w_y\right\rng\left\lng v^i,\frac{1}{2}w_z\right\rng.
\end{multline}
Applying   Cauchy-Schwarz   inequality and then,  using again both equalities  in  \eqref{eq:ewywz}, we conclude that  
\begin{align} \label{eq:esil3}
-\sum_{i=2}^n\theta_i({\bar c})\left\lng v^i,\frac{1}{2}w_y\right\rng\left\lng v^i,\frac{1}{2}w_z\right\rng &\leq    \sqrt{\sum_{i=2}^n\theta_i({\bar c})\lng v^i,w_y\rng^2}   \sqrt{\sum_{i=2}^n\theta_i({\bar c})\lng v^i,w_z\rng^2} \\                                      
                                                                                                                                                             &=-\left\lng v^1,\frac{1}{2}w_y\right\rng\left\lng v^1,\frac{1}{2}w_z \right\rng \notag. 
\end{align}
We are going to prove that Cauchy  inequality \eqref{eq:esil3} is  strict. For that, assume the contrary, i.e., that the last Cauchy inequality holds as equality. In this case,     there exists  $\alpha \neq 0$ such that 
$$
\left(\sqrt{\theta_2({\bar c})}\left\lng v^2,\frac{1}{2}w_y\right\rng, \ldots,  \sqrt{\theta_n({\bar c})}\left\lng v^n,\frac{1}{2}w_y\right\rng  \right)=\alpha \left(\sqrt{\theta_2({\bar c})}\left\lng v^2,-\frac{1}{2}w_z\right\rng, \ldots,  \sqrt{\theta_n({\bar c})}\left\lng v^n,-\frac{1}{2}w_z\right\rng  \right), 
$$
which implies  that $w_y+\alpha w_z$ is orthogonal to the set of vectors $\{v^2, \ldots, v^n \}$. Thus, since  $\{v^1, v^2,\dots,v^n\}$ is an orthonormal system,  $w_y+\alpha w_z$ is parallel to the vector $v^1$,  which is a contradiction due to vectors $v^1$,  $w_y$ and  $w_z$ being   L.I. Hence, \eqref{eq:esil3} holds strictly and combining it with \eqref{eq:fil3} we conclude that
$$
\sum_{i=2}^n\theta_i({\bar c})\left\lng v^i,\frac{1}{2}(w_y+w_z) \right\rng^2> \left\lng v^1,\frac{1}{2}(w_y+w_z) \right\rng^2, 
$$
and  \eqref{eq:mwywz} holds.  Therefore,  considering that $\frac{1}{2}(w_y+w_z) \in  ]u_y,u_z[$, we conclude that   $]u_y,u_z[ \not\subset {\mathcal L_{\bar c}}\cup-{\mathcal L_{\bar c}}$. Thus, by  using  the notation \eqref{eq:bl1-2},  we also have   $]u_y,u_z[ \not\subset ({\mathcal L}_{\bar c}\cup-{\mathcal L}_{\bar c})\cap\inte({\mathcal K})=[\varphi_A\le {\bar c}]$, and due to $u_y, u_z\in ({\mathcal L}_{\bar c}\cup-{\mathcal L}_{\bar c})\cap\inte({\mathcal K})=[\varphi_A\le {\bar c}]$, it follows that $[\varphi_A\le {\bar c}]$ is not 
			convex.\\
\noindent		
The {\it proof of the second statement}  follows  from $\mathcal K^*\subseteq\mathcal W^*$.  
\qed \end{proof}

\begin{remark}
	The dual of $\mathcal W$ in \eqref{eq:cw} can be expressed as 
	\begin{eqnarray}\label{eq-o-w}
		\begin{array}{rcl}
			\mathcal W^*=[(\mathcal K\cap\mathcal L_{\lambda_2})\cup (\mathcal K\cap-\mathcal
			L_{\lambda_2})]^*
			&=&(\mathcal K\cap\mathcal L_{\lambda_2})^*\cap (\mathcal K\cap-\mathcal L_{\lambda_2})^*\\
			&=&(\mathcal K^*+\mathcal L_{\lambda_2}^*)\cap(\mathcal K^*-\mathcal L_{\lambda_2}^*).
		\end{array}
	\end{eqnarray} 
\end{remark}

\begin{corollary}\label{cor-conv-sl}
	Suppose that $n\ge 3$ and $\lambda_2\le
	(\lambda_1+\lambda_3)/2$. If $\mathcal K\cap -\mathcal L_{\lambda_2}=\{0\}$ or $\mathcal
	K\cap \mathcal L_{\lambda_2}=\{0\}$, then $[\varphi_A\le c]$ is convex for all $c\notin (\lambda_2,\lambda_n)$.
\end{corollary}

\begin{proof}
	First note that if $n\ge 3$ and $\lambda_2\le (\lambda_1+\lambda_3)/2$, then $\theta_i(\lambda_2)\ge 1$ for any $i\ge 3$.  Define  the cone 
	\[
	 \mathcal L_{[v^2]^\perp}:=\lf\{x\in\R^n:~ ~\lng v^1,x\rng\ge\sqrt{\lng v^3,x\rng^2+\dots+\lng v^n,x\rng^2}\rg\}. 
	\]
 Note that $ \mathcal L_{[v^2]^\perp}$ is a self-dual Lorentz cone  as a subset of the subspace $[v^2]^\perp$.   Moreover,   considering that  $\theta_i(\lambda_2)\ge 1$ for any $i\ge 3$, we conclude $\mathcal L_{\lambda_2}\cap [v^2]^\perp\subset  \mathcal L_{[v^2]^\perp}$. Consequently,  taking into account  that  $ \mathcal L_{[v^2]^\perp}$ is a self-dual cone,  the cone  $\mathcal L_{\lambda_2}\cap [v^2]^\perp$ is    subdual  as a subset of the subspace $[v^2]^\perp$.   To simplify the notation,  denote by upper star (i.e., $^*$)
	the dual of a cone in $\mathbb R^n$ and by lower star (i.e., $_*$) the dual of a cone in $[v^2]^\perp$. Thus, using this notation we state
	\begin{equation} \label{eq;ssls}
	\mathcal
	L_{\lambda_2}^*=(\mathcal L_{\lambda_2}\cap [v^2]^\perp)_*
	\end{equation}
	 Indeed, since $v^2,-v^2\in\mathcal L_{\lambda_2}$, for 
	any $z\in\mathcal L_{\lambda_2}^*$, we have $\lng z,v^2\rng=0$ and hence $\mathcal L_{\lambda_2}^*\subseteq [v^2]^\perp$, which implies
	$\mathcal L_{\lambda_2}^*\subseteq (\mathcal L_{\lambda_2}\cap [v^2]^\perp)_*$. Conversely,  let  $u\in (\mathcal L_{\lambda_2}\cap [v^2]^\perp)_*$. Given $v\in\mathcal L_{\lambda_2}$,  take   $w\in\mathcal L_{\lambda_2}\cap [v^2]^\perp$ and  $t\in\R$ such that   $v=w+tv^2$. Hence,   $\lng u,v\rng=\lng u,w\rng\ge 0$, which implies that $u\in \mathcal L_{\lambda_2}^*$. Hence,  we conclude that 
	$(\mathcal L_{\lambda_2}\cap [v^2]^\perp)_*\subseteq\mathcal L_{\lambda_2}^*$, and \eqref{eq;ssls} is proved.  Next suppose $\mathcal K\cap -\mathcal L_{\lambda_2}=\{0\}$.  Hence,  by using the first equality in \eqref{eq-o-w} we obtain $\mathcal W^*=(\mathcal K\cap\mathcal L_{\lambda_2})^*$.  Therefore, considering  that  $\mathcal L_{\lambda_2}\cap [v^2]^\perp$ is    subdual  and \eqref{eq;ssls}, we obtain 
	 \[v^1\in\mathcal L_{\lambda_2}\cap[v^2]^\perp\subseteq (\mathcal L_{\lambda_2}\cap[v^2]^\perp)_*=\mathcal L_{\lambda_2}^*\subseteq
	(\mathcal K\cap\mathcal L_{\lambda_2})^*=\mathcal W^*.\] Hence, the result follows from Lemma \ref{Lem:Basic}.  The case $\mathcal K\cap {\cal L}_{\lambda_2}=\{0\}$ can be proved similarly.  
\qed \end{proof}
The next lemma is used in the proof of Proposition \ref{pr:multwo}, which is essential 
for proving Theorem \ref{th:best}. 
\begin{lemma}\label{lem:locnonconv}
	Let $n\ge 3$ and $B=B\tp\in\R^{n\times n}$. Let $\mu_1 \leq \mu_2 \leq  \ldots \leq \mu_n$ be  eigenvalues of the matrix $B$.  Assume that  $\mu_1\leq \mu_ 2<0<\mu_n$.  Then, for  any  vector ${\bar x}\in\R^n\setminus\{0\}$  such that $B{\bar x}\ne 0$ and  $\lf\lng B{\bar x},{\bar x}\rg\rng=0$,  and any number  $\delta>0$,   the    set 
	 $
	 \Xi\lf(B,{\bar x},\delta\rg):=\lf\{x\in\R^n:\lf\|x-{\bar x}\rg\|\le\delta,\textrm{ }\lng
	Bx,x\rng\le0\rg\}
	$
	 is not convex.
\end{lemma}
\begin{proof}
Since  $\mu_1=\min_{x\in \SP^{n-1}} q_B(x) <\max_{x\in \SP^{n-1}} q_B(x)= \mu_ n$, we can take   ${\bar x}\in\R^n\setminus\{0\}$  such that $B{\bar x}\ne 0$ and  $\lf\lng B{\bar x},{\bar x}\rg\rng=0$. Define the following vector subspace  ${\mc N}:=[\{u\in \R^n:~Bu=\mu  u,\textrm{ for some }\mu<0\}]$ of $\R^{n}$.   It follows from assumption (a) or (b)  that $\dim({\mc N})\ge2$.  To  simplify  the notation  set ${\bar y}:=B{\bar x}\ne0$. To proceed with the proof, we first need to prove that ${\mc N}\neq [\bar y]^\perp$.  Assume to the contrary that ${\mc N}= [\bar y]^\perp$. In this case, due to ${\bar y}=B{\bar x}$ and $B=B\tp$, the definition  of  $[\bar y]^\perp$ implies that $ \lf\lng Bv, {\bar x}\rg\rng=0$, for all $v\in {\mc N}$. Thus, the definition  of ${\mc N}$  implies   $ \lf\lng v, {\bar x}\rg\rng=0$, for all $v\in {\mc N}$,  which yields    ${\mc N}\subset [\bar x]^\pr:=\{v\in \R^{n}:~  \lf\lng v,{\bar x}\rg\rng=0\}$. Moreover,  considering that $\lf\lng {\bar y},{\bar x}\rg\rng=0$,  we also have ${\bar y} \in [\bar x]^\pr$. Hence,  we conclude that $[\bar{y}]+\mathcal{N}\subset [\bar x]^\pr$. If  ${\bar y}\notin {\mc N}$, then due to ${\bar y}\neq0$ and  ${\mc N}= [\bar y]^\perp$ we have $\dim ([\bar{y}]+\mathcal{N})=n$. Having that   $[\bar{y}]+\mathcal{N}\subset [\bar x]^\pr$, we obtain  ${\bar x} = 0$, which contradicts the assumption ${\bar x} \neq 0$. Hence, ${\bar y} \in {\mc N}= [\bar y]^\perp$, which also contradicts ${\bar y} \neq 0$. Therefore, ${\mc N}\neq [\bar y]^\perp$.  Thus,  we have  
	\[
	\dim(\mc N\cap [{\bar y}]^\perp)\geq \dim{\mc N}+ \dim{ [{\bar y}]^\perp} -  \dim\R^n \ge 2+(n-1)-n=1.
	\]
Hence, there exists a unit vector $a\in\mc N$ such that $\lng a,{\bar y}\rng=0$.   Since $\mathcal{N}\ne [\bar{y}]^\perp$, we can choose a sequence
of vectors  $\{a^k\}\subset \mc N$  such that $\lim_{k\to\infty}a^k=a$ and  $\lng a^k,{\bar y}\rng\ne 0$.  Let $\{u^1, u^2,\dots,u^n\}$ be an orthonormal system of eigenvectors of the matrix $B$  corresponding to the eigenvalues $\mu_1 , \mu_2 ,  \ldots , \mu_n$, respectively. Note that the spectral decomposition of $B$ implies
$
B=\sum_{i=1}^n\mu_i u^i(u^i)\tp.
$
Since  $\{a^k\}\subset \mc N$, we have $a^k=\sum_{i=1}^\ell\alpha_{k, i}  u^i$, where $2\leq \ell=\dim(\mc N)<n$ and $\mu_1,  \ldots , \mu_\ell$ are the negative eigenvalues  of $B$. Thus, 
$
\lng Ba^k,a^k\rng=\sum_{i=1}^\ell\alpha^2_{k, i}  \mu_i<0.
$
For proceeding with the proof,  we define 
\[
 p^k:={\bar x}+t_ka^k, \qquad t_k:=-2\f{\lng a^k,{\bar y}\rng}{\lng Ba^k,a^k\rng}.
\] 
Then,
	$\lng Bp^k,p^k\rng=0$ and,  due to $\lng a,{\bar y}\rng=0$ and $\lim_{k\to\infty}a^k=a$, we have  $\lim_{k\to\infty}p^k={\bar x}$. Hence, if $k$ is sufficiently large, then for any  $\delta>0$ arbitrary but fixed, we have  $p^k\in\Xi\lf(B,{\bar x},\delta\rg)$. 
	For such an $k$,  after some simple algebraic manipulations we conclude 
	\[\lf\lng B\lf(\f{{\bar x}+p^k}2\rg),\f{{\bar x}+p^k}2\rg\rng=-\f{\lf\lng a^k,{\bar y}\rg\rng^2}{\lng Ba^k,a^k\rng}>0.
	\]
	Hence, ${\bar x},p^k\in\Xi\lf(B,{\bar x},\delta\rg)$, but $({\bar x}+p^k)/2\notin\Xi\lf(B,{\bar x},\delta\rg)$. Therefore, $\Xi\lf(B,{\bar x},\delta\rg)$ is not
	convex.   
\qed \end{proof}
\begin{proposition}\label{pr:multwo}
	Let $n\ge 3$ and $A=A\tp\in\R^{n\times n}$  is a nonsingular matrix.   Suppose that $q_A$  is    not constant  and   $\lambda_1 \leq \lambda_2 \leq  \ldots \leq \lambda_n$ are eigenvalues of $A$.  If $q_A$  is  quasiconvex, then the following conditions hold:
\begin{enumerate}
		\item[(i)] $\lambda_1 < \lambda_2$;
		\item[(ii)] either   $\lambda_2\leq \min_{x\in {\overline{\cal C}}} q_A(x)$  or  $\max_{x\in {\overline{\cal C}}} q_A(x)\leq\lambda_2$. 
	\end{enumerate}	
\end{proposition}
\begin{proof}
 Suppose by contradiction that one of the  following two conditions holds:
	\begin{enumerate}
		\item[(a)] $\lambda_1= \lambda_2$;
		\item[(b)]  $\min_{x\in {\overline{\cal C}}} q_A(x) <\lambda_2<\max_{x\in {\overline{\cal C}}} q_A(x)$. 
	\end{enumerate}
First of all, note that due to  $q_A$ not being constant, we have   $\lambda_1\leq\min_{x\in {\overline{\cal C}}} q_A(x) <\max_{x\in {\overline{\cal C}}} q_A(x)\leq \lambda_ n$, where ${\overline{\cal C}}$ is defined in \eqref{eq:ccb}.  Hence, we can  take a $\mu \in \R$ such that $\mu \neq \lambda_i$ for all $i=1, \ldots n,$ and  satisfying 
\begin{equation} \label{eq:cdfff}
 \lambda_1=\lambda_ 2\leq\min_{x\in {\overline{\cal C}}} q_A(x)< \mu  <\max_{x\in {\overline{\cal C}}} q_A(x)\leq \lambda_n, 
\end{equation}
if the   condition (a) holds. Otherwise,  if the condition (b) holds, we take a $\mu \in \R$  satisfying 
\begin{equation} \label{eq:cdsss}
\lambda_1\leq\min_{x\in {\overline{\cal C}}} q_A(x)<\lambda_ 2<\mu< \max_{x\in {\overline{\cal C}}} q_A(x)\leq \lambda_n.
\end{equation}
	Then, any of the conditions \eqref{eq:cdfff} or \eqref{eq:cdsss} implies that  $\pm(A-\mu I_n)$ is not $\mc K$-copositive. Since, the matrix $A-\mu I_n$ is not $\mc K$-copositive,
	there exists a $p\in\mc K$ such that $\lng Ap,p\rng<\mu\|p\|^2$. Hence, there
	exist also an $u\in\inte(\mc K)$ sufficiently close to $p$ such that $\lng Au,u\rng<\mu\|u\|^2$. Similarly, since $-(A-\mu I_n)=\mu I_n-A$
	is not $\mc K$-copositive, there exists a $v\in\inte(\mc K)$ such that $\lng Av,v\rng>\mu\|v\|^2$. Therefore, by continuity,
	there exists a $t\in ]0,1[$ such that $\lng A{\bar x},{\bar x}\rng=\mu\|{\bar x}\|^2$, where ${\bar x}=(1-t)u+tv\in\inte(\mc K)$.  Letting $B=A-\mu I$, its eigenvalues  are given by $\mu_i:=\lambda_i-\mu$, for $i=1, 2, \ldots, n$. Thus,  we conclude from \eqref{eq:cdfff} and \eqref{eq:cdsss} that  
\begin{equation} \label{eq:cdfffa}
 \mu_1=\mu_ 2<0 < \mu_n,   \quad \mbox{or} \qquad \mu_1<\mu_ 2<0<\mu_n,  
\end{equation}
if   the condition (a) or (b) holds, respectively. Since  $B{\bar x}\ne 0$ and  $\lng B{\bar x},{\bar x}\rng=0$, we conclude from Lemma~\ref{lem:locnonconv} that,  for all $\delta>0$,  the set 
	$\Xi\lf(B,{\bar x},\delta\rg):=\lf\{x\in\R^n:\lf\|x-{\bar x}\rg\|\le\delta,\textrm{ }\lng Bx,x\rng\le0\rg\}, $
	  is not convex. Hence, there exists an $s\in ]0,1[$ and
	$a^0,a^1\in\Xi\lf(B,{\bar x},\delta\rg)$ such that $a^s:=(1-s)a^0+sa^1\notin\Xi\lf(B,{\bar x},\delta\rg)$. Thus, owing  to the ball of
	centre ${\bar x}$ and radius $\delta$  is   convex, $a^s\notin \Xi\lf(B,{\bar x},\delta\rg)$ implies 
	$\lng Aa^s,a^s\rng-\mu\|a^s\|^2=\lng Ba^s,a^s\rng>0$. On the other hand, since $a^0,a^1\in\Xi\lf(B,{\bar x},\delta\rg)$, we
	have $\lng Aa^i,a^i\rng-\mu\|a^i\|^2=\lng Ba^i,a^i\rng\le0$, for $i\in\{0,1\}$. Furthermore, if $\delta$ is sufficiently
	small, then considering that  ${\bar x}\in\inte(\mc K)$, we have $a^0,a^1\in\inte\mc K$. Hence,	$a^0,a^1\in[\varphi_A\le\mu]$ and
	$a^s\notin[\varphi_A\le\mu]$. By using Corollary~\ref{cor:cor}, this contradicts the spherical quasiconvexity of $A$.  
 \qed \end{proof}
\begin{lemma} \label{lem:Copositive}
 Let  $A \in\R^{n\times n}$ and $\lambda,c\in\R$ such that $\lambda\leq c$.  If  the matrix $\lambda  I_n-A$ is ${\cal K}$-copositive,   then  $[\varphi_A\leq c]=\inte ({\cal K})$. As a consequence, the set $[\varphi_A\leq c]$ is convex.
\end{lemma}
\begin{proof}
Since  $\lambda\leq c$, we have  $\langle A x,x \rangle-c\|x\|^2\leq  \langle A x,x \rangle-\lambda \|x\|^2=  \langle (A-\lambda I_n) x,x
\rangle$. Thus, considering that $\lambda  I_n-A$ is ${\cal K}$-copositive, we conclude that $\langle A x,x \rangle-c\|x\|^2\leq 0$, for all $x\in
\inte ({\cal K})$, which implies that  $[\varphi_A\leq c]=\{x\in  \inte ({\cal K}):~   \langle A x,x \rangle-c \|x\|^2\leq 0\}=\inte ({\cal K})$.   
\qed \end{proof}
\begin{theorem}\label{th:best}
Let $n\geq 3$, $k\ge1$, $A=A\tp\in\R^{n\times n}$ and  $\{v^1, v^2,\dots,v^n\}$ be an orthonormal system of eigenvectors of $A$  corresponding to
the eigenvalues $\lambda_1=\dots=\lambda_k<\lambda_{k+1}\leq \ldots \leq \lambda_n$, respectively.  Then, we have the following
statements:
\begin{enumerate}
	\item[(i)] If   $q_A$  is  quasiconvex and not constant, then $k=1$. 
	\item[(ii)]  If   $q_A$  is  quasiconvex and not constant, then either   $\lambda_2\leq \min_{x\in {\overline{\cal C}}} q_A(x)$  or  $\max_{x\in {\overline{\cal C}}} q_A(x)\leq\lambda_2$. 
	\item[(iii)] Suppose that $k=1$ and $\lambda_2I_n-A$ is $\mc K$-copositive. Then, $q_A$ is spherically quasiconvex if and only if 
		$v^1 \in {\cal W}^*\cup-{\cal W}^*$. In particular if $v^1 \in {\cal K}^*$, then $q_A$ is spherically quasiconvex.
	\item[(iv)] Suppose that  ${\cal K}^*= {\cal K}$  and $v^1\in {\cal K}$. Then, $q_A$ is spherically
		quasiconvex if and only if $k=1$ and $\lambda_2 I_n-A$ is ${\cal K}$-copositive. 
\end{enumerate}
\end{theorem}
\begin{proof}
	Items (i) and (ii) follow from Proposition \ref{pr:multwo}. Item (iii) follows from Lemma \ref{Lem:Basic},  Lemma \ref{lem:Copositive} and Corollary~\ref{cor:cor}.  
	Next, we prove item (iv). Suppose that $k=1$,  ${\cal K}^*= {\cal K}$ and $v^1\in {\cal K}$. If $\lambda_2 I_n-A$ 
	is ${\cal K}$-copositive, then the result follows from item (iii). If $q_A$ is spherically quasiconvex,
	then item (i) implies that $k=1$ and   item (ii) implies that either $A-\lambda_2 I_n$ or $\lambda_2 I_n-A$ is ${\cal K}$-copositive. If  
	$A-\lambda_2 I_n$ is ${\cal K}$-copositive, then $v^1\in {\cal K}$ implies 
	$\lng (A-\lambda_2 I_n)v^1,v^1\rng\ge0$. Hence, $\lambda_1\ge\lambda_2$, which contradicts $k=1$.
	Therefore, this case cannot hold and we must have $\lambda_2 I_n-A$ is ${\cal K}$-copositive.
 \qed \end{proof}
 The next corollary follows  by  combining  Lemma~\ref{lem:Copositive} and  Corollary~\ref{cor-conv-sl}. 
\begin{corollary}\label{cor-best}
Let $n\geq 3$, $A=A\tp\in\R^{n\times n}$ and   $\lambda_1<\lambda_{2}\leq \ldots \leq \lambda_n$  the eigenvalues of $A$.   Suppose that $\lambda_2\le
	(\lambda_1+\lambda_3)/2$ and $\lambda_2I_n-A$ is $\mc K$-copositive. If $\mathcal K\cap -\mathcal L_{\lambda_2}=\{0\}$ or $\mathcal
	K\cap \mathcal L_{\lambda_2}=\{0\}$, then $q_A$ is spherically quasiconvex.
\end{corollary}
In the following   two  theorems    we   present classes of quadratic quasi-convex functions defined  in
spherically subdual convex sets, which
include as particular instances \cite[Examples~18 and  ~19]{FerreiraNemethXiao2018}.
\begin{theorem} \label{th:countergg}
Let $n\geq  3$,     $A=A\tp\in\R^{n\times n}$ and $\{v^1, v^2,\dots,v^n\}$ be an orthonormal system of eigenvectors of the matrix  $A$  corresponding to the eigenvalues $\lambda_1 , \lambda_2 ,  \ldots , \lambda_n$ , respectively.  Moreover,   assume that   $\lambda:=\lambda_1$, $\mu:=\lambda_2=\ldots= \lambda_{n-1}$,  $\eta:=\lambda_n$ and
\begin{equation} \label{eq:ciqc}
 v^1-\sqrt{\frac{\eta-\mu}{\mu-\lambda}}|v^{n}|^{\cal K}\in {\cal K}^*,   \qquad  \lambda<\mu<\eta, 
\end{equation}
where $ |\cdot |^{\cal K}$ is defined in  \eqref{eq;npav}.	 Then,  the quadratic function    $q_A$ is spherically quasi-convex.
\end{theorem}
\begin{proof}
By using the spectral decomposition of $A$, we have
\begin{equation} \label{eq:equivA}
A=\sum_{i=1}^n\lambda_i v^i(v^i)\tp=\lambda v^1(v^1)\tp+ \sum_{j=2}^{n-1} \mu v^j(v^j)\tp+\eta v^n(v^n)\tp.
\end{equation}
Hence,  for all   $x\in{\cal K}$,  by using $\|x\|^2=\sum_{i=1}^n\lng v^i,x\rng^2$ and \eqref{eq:equivA}, after some calculations  we obtain
\begin{equation}\label{eq:ceq}
  \langle A x,x \rangle-\mu\|x\|^2= (\mu-\lambda) \left[\frac{\eta-\mu}{\mu-\lambda}\langle v^n, x\rangle^2-\langle v^1, x\rangle^2 \right].
\end{equation}
To procced with the proof we  note that  \eqref{eq;npav} implies that  $|v^n|^{\cal K}\in{\cal K}+{\cal K}^*$ and,  owing to  ${\cal K} \subseteq{\cal K}^*$, we conclude that   $|v^n|^{\cal K}\in {\cal K}^*$. Thus, \eqref{eq:ciqc} implies that 
\begin{equation}\label{eq:sdi}
0 \leq  \sqrt{\frac{\eta-\mu}{\mu-\lambda}} \lng |v^n|^{\cal K},x\rng\leq \langle v^1, x\rangle, \qquad \forall ~ x\in{\cal K}.
\end{equation}
Hence,  for all   $x\in{\cal K}$, the last inequality  yields 
\begin{equation}\label{eq:ie1}
\frac{\eta-\mu}{\mu-\lambda}\langle v^n, x\rangle^2 -\langle v^1, x\rangle^2 \leq \frac{\eta-\mu}{\mu-\lambda}\left[\lng v^n,x\rng^2-\lng|v^n|^{\cal K},x\rng^2\right] =\frac{\eta-\mu}{\mu-\lambda}\lng v^n+|v^n|^{\cal K},x\rng\lng v^n-|v^n|^{\cal K},x\rng.
\end{equation}
On the other hand, by using $|{v^{n}}|^{\cal K}={\rm P}_{\cal K}(v^{n})+{\rm P}_{{\cal K}^*}(-v^{n})$,  ${v^{n}}={\rm P}_{\cal K}({v^{n}})-{\rm P}_{{\cal
K}^*}(-{v^{n}})$, $P_{\cal K}(v^n)\in K\subseteq K^*$, we obtain $\lng v^n+|v^n|^{\cal K},x\rng\lng v^n-|v^n|^{\cal K},x\rng=-4\lng P_{\cal K}(v^n),x\rng\lng P_{{\cal K}^*}(-v^n),x\rng\leq 0$, for all   $x\in{\cal K}$. Thus, due to $\lambda<\mu<\eta$,  the previous inequality together  with  \eqref{eq:ie1}  implies 
\begin{equation}\label{eq:ie2}
	\frac{\eta-\mu}{\mu-\lambda}\langle v^n, x\rangle^2 -\langle v^1, x\rangle^2 \leq 0, \qquad \forall ~ x\in{\cal K}.
\end{equation}
	Thus,  considering  that $\lambda<\mu$,  the combination of   \eqref{eq:ceq} with   \eqref{eq:ie2}, implies that  $\mu {\rm I_n}-A$ is ${\cal
	K}$-copositive.  Taking into account that  $|v^n|^{\cal K}\in {\cal K}^*$, \eqref{eq:ciqc} implies $v^1\in{\cal K}^*$.
	Therefore, we can apply the item (iii) of Theorem~\ref{th:best}   to conclude that $q_A$ is  spherically quasi-convex. 
 \qed \end{proof}
 \begin{remark}
	The subduality of $\mathcal K$ is essential for the proof of 
Theorem \ref{th:countergg}. Indeed, suppose that the cone $\mathcal K$ is
not-selfdual and $\mathcal K^*$ is subdual (for example take a pointed closed convex cone which contains
the nonnegative ortant in its interior). Then, $\mathcal K^*\subsetneqq\mathcal K$, hence
$\mathcal K$ is
not subdual. Choose $\mathcal K$ such that $v^n\in\mathcal K\setminus\mathcal K^*$. It follows 
that $|v^n|^{\mathcal K}=v^n\notin\mathcal K^*$ and therefore the first inequality 
in \eqref{eq:sdi} fails. Hence, the proof of Theorem \ref{th:countergg} fails.
\end{remark}
The following example satisfies the assumptions of Theorem~\ref{th:countergg}.
\begin{example} \label{ex:countergg}
Letting  $ {\cal K}=\R^{n}_{+}$  and   $\lambda<(\lambda+\eta)/2<\mu<\eta$, the unit vectors    $v^1=(e^1+e^n)/\sqrt{2}$, $v^2=e^2$, $\ldots$,
 $v^{n-1}=e^{n-1}, v^n=(e^1-e^n)/\sqrt2$ are pairwise orthogonal  and  satisfy the condition  \eqref{eq:ciqc}. Now,  taking $\cal{K} = \cal{L}$
and denoting  $v^n = \lf((v^n)_1, (v^n)^2\rg)$, by using Lemma~\ref{l:projude},  \eqref{eq:ciqc} can be written as
$$
		v^1-\sqrt{\frac{\eta-\mu}{\mu-\lambda}}\f1{\|(v^n)^2\|}\Big{(} \max\left(|(v^n)_1|,\|{(v^n)^2}\|\right)\|(v^n)^2\|,~\min\left(|(v^n)_1|,\|{(v^n)^2}\|\right)\sgn((v^n)_1)(v^n)^2\Big{)}\in {\cal K},
$$
and $\lambda<\mu<\eta$.
 The vectors  $v^1=(e^1+e^n)/\sqrt{2}, v^2=e^2, ~ \ldots, ~ v^{n-1}=e^{n-1}, v^n=(-e^1+e^n)/\sqrt2$  are pairwise orthogonal  and   satisfy the last inclusion.
\end{example}
\begin{theorem} \label{th:CountExMulEngfm}
Let $n\geq  3$,     $A=A\tp\in\R^{n\times n}$ and $\{v^1, v^2,\dots,v^n\}$ be an orthonormal system of eigenvectors of $A$  corresponding to the
eigenvalues $\lambda_1 , \lambda_2 ,  \ldots , \lambda_n$, respectively, such that $v^1\in\inte({\cal K}^*)$.  Let
\begin{equation}\label{ineq:alpha}
	\alpha:=\min\{ \langle v^1, y\rangle^2:~y\in \SP^n\cap{\cal K}\}>0,\quad
	\eta:=\max\lf\{\frac{\langle v^3,y\rangle^2+...+\langle v^n,y\rangle^2}{\langle v^1,y\rangle^2}:~y\in S^n\cap{\cal K}\rg\}>0.
\end{equation}
Assume that
\begin{equation} \label{eq:CondEngfm}
 \qquad \lambda_1<\lambda_2\leq \cdots \leq \lambda_n \leq\lambda_2+\delta(\lambda_2-\lambda_1), \qquad \delta\in\{\alpha,1/\eta\}.
\end{equation}
Then,    $\lambda_{2} {\rm I_n}-A$ is ${\cal K}$-copositive. Consequently,  the quadratic function $q_A$ is spherically quasi-convex.
\end{theorem}
\begin{proof}
Note that the spectral decomposition of $A$ implies $A=\sum_{i=1}^n\lambda_i v^i(v^i)\tp$. Thus, considering that  $\|x\|^2=\sum_{i=1}^n\lng v^i,x\rng^2$, for all   $x\in{\cal K}$,  we conclude that
\begin{equation} \label{eq:equivnormA}
  \langle A x,x \rangle-\lambda_{2}\|x\|^2= \sum_{i=1}^n(\lambda_i-\lambda_{2})\langle v^{i}, x\rangle^2.
\end{equation}
	Since  \eqref{eq:CondEngfm}   implies   $\lambda_{2}-\lambda_1>0$ and   $0\leq \lambda_{j}-\lambda_{2}\leq \lambda_n-\lambda_{2}$,  for all $j=3, \ldots, n$, it follows   from  \eqref{eq:equivnormA}  that 
$$
  \langle A x,x \rangle-\lambda_{2}\|x\|^2 \leq (\lambda_{2}-\lambda_1)\left[ \frac{\lambda_n-\lambda_{2}}{\lambda_{2}-\lambda_1}\left( \langle
  v^{3}, x\rangle^2+ \cdots +   \langle v^n, x\rangle^2\right)-\langle v^1, x\rangle^2 \right].
$$
Since the  second equality  in \eqref{ineq:alpha} implies  $ \langle  v^{3}, x\rangle^2+ \cdots +   \langle v^n, x\rangle^2\leq \eta \langle v^1, x\rangle^2$, the latter inequality becomes
\begin{equation} \label{eq:coprfm}
 \langle A x,x \rangle-\lambda_{2}\|x\|^2 \leq (\lambda_2-\lambda_1)  \left[\lf(\eta\frac{\lambda_n-\lambda_{2}}{\lambda_{2}-\lambda_1}-1\rg)\langle v^1,x\rangle^2\right].
\end{equation}
First assume  that $\delta=1/\eta$. Thus,  the last inequality in  \eqref{eq:CondEngfm} implies $\eta (\lambda_n-\lambda_{2})/(\lambda_{2}-\lambda_{1})\leq 1$, which combined with     \eqref{eq:coprfm} yields 
\begin{equation} \label{eq;cpla}
  \langle A x,x \rangle-\lambda_{2}\|x\|^2 \leq 0.
\end{equation} 
Next, assume  that $\delta=\alpha$. First of all,  note that $\sum_{i=3}^n\lng v^i,y\rng^2\leq \sum_{i=1}^n\lng v^i,y\rng^2=\|y\|^2=1$, for all $y\in S^n$. Thus, using \eqref{ineq:alpha},  we conclude that 
\begin{eqnarray*}
	\eta=\max\lf\{\frac{\langle v^3,y\rangle^2+...+\langle v^n,y\rangle^2}{\langle v^1,y\rangle^2}:~ y\in S^n\cap K\rg\}
	\le \max\lf\{\frac1{\langle v^1,y\rangle^2}:y\in S^n\cap K\rg\}=\frac1\alpha.
\end{eqnarray*}
 Hence,   it follows from  \eqref{eq:coprfm}  that 
\begin{equation} \label{eq:sccp}
  \langle A x,x \rangle-\lambda_{2}\|x\|^2 \leq  (\lambda_2-\lambda_1)\left[\lf(\frac1\alpha\frac{\lambda_n-\lambda_{2}}{\lambda_{2}-\lambda_1}-1\rg)\langle v^1,x\rangle^2\right].
  \end{equation} 
Due  to  $\delta=\alpha$, the last inequality in  \eqref{eq:CondEngfm}  implies $ (\lambda_n-\lambda_{2})/[\alpha (\lambda_{2}-\lambda_{1})]\leq 1$, which together with \eqref{eq:sccp} also implies  \eqref{eq;cpla}.  Hence, we conclude  that
$\lambda_{2} {\rm I_n}-A$ is ${\cal K}$-copositive. Therefore, since $v^1\in {\cal K}^*$ and it  is an eigenvector  of
$A$ corresponding to the eigenvalue $\lambda_1$,  by applying item (iii) of Theorem~\ref{th:best}, we can conclude that the function $q_A$ is
spherically quasi-convex.  
 \qed \end{proof}
 \begin{remark}
	 The numbers $\alpha$ and $\eta$ have an analytical motivation for the
required copositivity. The number $\alpha$ is the cosine square of the maximal 
angle between the vector $v^1$ and the cone $\mathcal K$ (i.e., the maximal Euclidean angle between the
first eigenvector and the nonzero vectors of the cone). For the number $\eta$ consider the affine
hyperspace $\mathcal H=\{x\in\R^n:~\langle v^1,x\rangle=1\}$ and the $(n-2)$-dimensional linear subspace 
$\mathcal M=\{x\in\R^n:~\langle v^1,x\rangle=0\rangle,\textrm{ }\langle v^2,x\rangle=0\}$. Then, 
$\eta$ is the maximal length of a vector in the orthogonal projection of the slice
$\mathcal K\cap\mathcal H$ onto $\mathcal M$.
 \end{remark}
In the following we present an example   satisfying the assumptions of Theorem~\ref{th:CountExMulEngfm}.
\begin{example}\label{Ex:CountExMulEngfm}
 Letting ${\cal L}$ the Lorentz cone,  the vectors $v^i=e^i$, for all $i\in\{1, \ldots, n\}$ and  the eigenvectors $\lambda_1<\lambda_2\leq \ldots \leq
\lambda_n<\lambda_2 +(1/2)(\lambda_2-\lambda_1)$,   the  condition \eqref{eq:CondEngfm} is satisfied. In this case $\alpha=1/2$.
\end{example}

\begin{corollary} \label{cr:TwoEng}
 Let $n\ge 3$ and  $\{v^1, v^2,\dots,v^n\}$ be an orthonormal system of eigenvectors of  $A=A\tp\in\R^{n\times n}$ corresponding to the eigenvalues $\lambda_1 , \lambda_2 ,  \ldots ,\lambda_n$ , respectively.  Assume that $\lambda_1=:\lambda <\mu:=\lambda_2=\dots=\lambda_n$. If $v^1 \in {\cal K}^*$, then  $q_{A}$ is spherically quasi-convex.
\end{corollary}
\begin{proof}
 Using the spectral decomposition of $A$, we have
\begin{equation} \label{eq:TwoEng}
	A= \lambda v^1(v^1)\tp+\sum_{j=2}^n\mu v^j(v^j)\tp.
\end{equation}
Since  $\|x\|^2=\sum_{i=1}^n\lng v^i,x\rng^2$,    for all   $x\in \R^{n}$,  by using \eqref{eq:TwoEng}  and $\lambda <\mu$, we  obtain 
\begin{equation} \label{eq:fetw}
\mu\|x\|^2-\langle A x,x \rangle= (\mu-\lambda)\langle v^{1}, x\rangle^2\geq 0, \qquad \forall x\in \R^{n}.
\end{equation} In particular,   \eqref{eq:fetw} implies that $\mu I_n-A$ is ${\cal K}$-copositive. Thus, since $ v^1 \in
{\cal K}^*$, by applying item (iii) of Theorem~\ref{th:best} with $\lambda_2=\mu$ we can conclude that the function $q_A$ is spherically
quasi-convex. 
 \qed \end{proof}
In the next example we show how to generate  matrices  satisfying the assumptions of Corollary~\ref{cr:TwoEng}
and consequently generate  spherically quasi-convex functions on  spherically subdual convex sets.
\begin{example}
	The Householder matrix associated to  $v\in\inte({\cal K}^*)$ is defined by $H:={\rm I_n}-2vv^{T}/\|v\|^2$. We know that   $H$ is a symmetric and nonsingular matrix. Furthermore,  $Hv=-v$ and  $Hu=u$ for all $u\in {\cal S}$, where  $ {\cal S}:=\{ u\in \R^n~:~ \langle v, u \rangle=0\}$.  Since the
dimension of ${\cal S}$ is $n-1$, then  we have  $1$ and $-1$ are  eigenvalues of $H$  with multiplicities $n-1$ and $1$, respectively.
Moreover,  considering that $v\in\inte({\cal K}^*)$, Corollary~\ref{cr:TwoEng} implies that   $q_{H}(x)=\langle Hx,x\rangle$ is spherically
quasi-convex.
\end{example}
\section{Spherically Quasi-Convex Quadratic Functions on the  Spherical  Lorentz Convex Set} \label{sec:qcqflc}
 In this section we present a  condition partially characterizing the  spherical   quasi-convexity of
 quadratic functions on  spherically convex  sets associated to the    Lorentz  cone. First,  we remark that for the Lorentz cone ${\cal L}$, since by Lemma \ref{lorcop} we have a characterization of ${\cal L}$-copositive matrices,
item (iii) of Theorem~\ref{th:best} provides a more general result than Theorem \ref{th:CountExMulEngfm}.
\begin{proposition}
Let ${\cal L}$ be the Lorentz cone, $n\geq  2$,  $A=A\tp\in\R^{n\times n}$,  $\lambda_1\le\lambda_2\leq...\leq\lambda_n$ be  the eigenvalues  of $A$, $v^1$ be   an eigenvector  of $A$    corresponding to   $\lambda_1$ and $J=\diag(1,-1,\dots,-1)\in\R^{n\times n}$.
	If $v^1\in {\cal L}$ and there exists an $\rho\ge0$ such that $\lambda_2I_n-A-\rho J$ is positive semidefinite, then $q_A$ is
	spherically quasiconvex.
\end{proposition}
\begin{proof}
If  there exists an $\rho\ge0$ such that $\lambda_2I_n-A-\rho J$ is positive semidefinite, then it follows from Lemma~\ref{lorcop}  that $\lambda_2I_n-A$ is a ${\cal L}$-copositive matrix. Therefore, considering that $v^1\in {\cal L}={\cal L}^*$ and it  is an eigenvector  of
$A$ associated  to the eigenvalue $\lambda_1$,  by applying item (iii) of Theorem~\ref{th:best}, we  conclude that $q_A$ is
spherically quasi-convex.  
\qed \end{proof}
 The next result  is a version  of   \cite[Theorem 20]{FerreiraNemethXiao2018} for this  cone.
\begin{theorem}\label{thm:2neg-v}
 Let $n\ge 3$ and  $\{v^1, v^2,\dots,v^n\}$ be an orthonormal system of eigenvectors of  $A=A\tp\in\R^{n\times n}$ corresponding to the eigenvalues $\lambda_1 , \lambda_2 ,  \ldots ,\lambda_n$ , respectively.  Assume that $\lambda_1=:\lambda <\mu:=\lambda_2=\dots=\lambda_n$.  Then, $q_{A}$ is a spherically quasi-convex function if and only if $v^1\in {\cal L}$.  
\end{theorem}
\begin{proof}
If  there exists an eigenvector  of $A$ corresponding  to the smaller eigenvalue belonging to   ${\cal L}$, then Corollary~\ref{cr:TwoEng} implies
that $q_{A}$ is spherically quasi-convex.  Conversely, assume that $q_{A}$ is spherically quasi-convex.  Thus, by using the spectral decomposition of $A$, we have
\begin{equation} \label{eq:VLVT}
	A= \lambda v^1(v^1)\tp+\sum_{j=2}^n\mu v^j(v^j)\tp.
\end{equation}
We can also assume,  without loss of generality,  that the number $v^1_ 1\geq 0$.
Let $x\in \p{\cal L}\setminus\{0\}$ and  note that    $y=2x_1e^1-x\in \p{\cal L}\setminus\{0\}$.  Since $\sum_{i=1}^nv^i(v^i)\tp=I_n$ (i.e., the spectral
decomposition of $I_n$) and $\langle x, y\rangle=0$, \eqref{eq:VLVT} implies that
\begin{equation}\label{eq:l}
	\lng Ax,y\rng=\lf\lng\lf[\mu\sum_{i=1}^nv^i(v^i)\tp+(\lambda-\mu)(v^1)(v^1)\tp\rg] x,y\rg\rng=(\lambda-\mu) \lng v^1,x\rng\lng v^1,y\rng.
 \end{equation}
 Since  $x, y\in {\cal L}$,  $\langle x, y\rangle=0$ and ${\cal L}$ is a self-dual cone, it follows  from  Corollary~\ref{cor:qpcf} that  $ \lng Ax,y\rng\leq 0$.
 Thus, considering that $\lambda <\mu$ and  $y=2x_1e^1-x$,  \eqref{eq:l} yields
\begin{equation}\label{ineq:l}
	0\leq  \lng v^1,x\rng\lng v^1,y\rng=\lng v^1,x\rng [2v^1_1x_1-\lng v^1,x\rng].
\end{equation}
On the other hand, due to  $x\in{\cal L}$, we have $x^1\ge0$. Thus,   since  $v_1^1\ge0$, if $\lng v^1,x\rng<0$, then we have $\lng
v^1,x\rng[(2v^1_1x_1-\lng v^1,x\rng]<0$, which
contradicts \eqref{ineq:l}. Hence $\lng v^1,x\rng\ge0$, where $x$ can be chosen arbitrarily in $\p{\cal L}\setminus\{0\}$. Therefore,
$v^1\in {\cal L}$ and the proof is complete. 
 \qed \end{proof}
\section{Perspectives and Open Problems} \label{sec:pop}
First of all, we note that  for all our classes of  spherically quasi-convex quadratic  functions $q_A$ on
the spherically subdual convex set  ${\cal C}=\SP^{n-1}\cap\inte({\cal K})$, the matrix $A$ has  the smallest eigenvalue  with  multiplicity one
and the associated eigenvector belongs  to the dual ${\cal K}^*$ of the subdual cone ${\cal K}$. We conjecture that this condition is necessary
and sufficient to characterize spherically quasi-convex quadratic  functions. It is worth
noting that for selfdual  cones  such a characterization can be obtained by proving that,  under the assumption  $v^1\in {\cal K}$,  the matrix  $\lambda_2 I_n-A$ is ${\cal K}$-copositive, where  $\{v^1, v^2,\dots,v^n\}$ is an orthonormal system of eigenvectors of $A$  corresponding to the eigenvalues $\lambda_1<\lambda_{2}\leq \ldots \leq \lambda_n$, respectively. We also remark  that, in Theorem~\ref{thm:2neg-v} we present a partial characterizations of spherically
quasi-convex quadratic functions on the  spherical Lorentz convex set.  However, the general question remains open even for this specific set.

Although efficient algorithms for solving intrinsically unconstrained problems posed on the 
whole  sphere are well known (see \cite{Hager2001, HagerPark2005,
Smith1994,So2011,Zhang2003,Zhang2012}),  an even more challenging problem is to develop  efficient algorithms for constrained quadratic optimization problems on spherically convex sets.  As far as we know, the first algorithm  for a special problem on  this subject has  recently appeared in \cite[Example 5.5.2]{Lange2016}; see also \cite[Section 10]{BauschkeBuiWang2018}.

 Minimizing a quadratic function on the intersection of the Lorentz cone with the sphere is particularly relevant, since the nonnegativity of the minimum value is equivalent to the Lorentz-copositivity of the corresponding matrix, see \cite{LoewySchneider1975,GajardoSeeger2013};  see also  \cite[Example 5.5.2]{Lange2016} and  \cite[Section 10]{BauschkeBuiWang2018}.  In general, replacing the Lorentz cone with an
arbitrary closed convex cone $K$ leads to the more general concept of ${\cal K}$-copositivity. By considering the intrinsic geometrical properties of the sphere, interesting perspectives for detecting the  general copositivity of matrices emerge.

Note that isometries map geodesics onto geodesics and   also preserve convex sets.  Moreover,  the
composition of convex and quasi-convex  functions  with isometries are also convex and
quasi-convex  functions, respectively.   Hence, all concepts studied in this paper are preserved
by isometries of the sphere. On the other hand, considering that the Euclidean space and the
sphere are not  isometric, there is no clear relationship  between concepts  of convexity and
quasi-convexity  of these  geometric spaces. On the other hand,  in principle, it is possible to
extend concepts similar to those studied in this paper to any submanifolds of the Euclidean space,
which naturally depend on curvatura/metric.   As a consequence, the following question arises: It
is possible to unify these concepts or part of them in a common framework? This is an interesting question that deserves to be investigated.

\section{Conclusions} \label{sec:Conclusions}
In \cite{FerreiraNemethXiao2018, FerreiraIusemNemeth2014, FerreiraIusemNemeth2013,  FerreiraNemeth2017} we studied some intrinsic properties of the spherically  convex sets and functions. In the present paper we showed further developments of this topic. In particular,  many  of the results obtained  in the prevous papers are related to the conditions implying spherical  quasi-convexity of quadratic functions on the spherical positive orthant. In the present paper these results are generalized to  subdual cones.  As far as we know this is the
pioneer study of spherically quasi-convex quadratic functions on  spherically  subdual convex sets.  As stated in Section~\ref{sec:pop},  there are still interesting questions to be answered in this topic,   we foresee further progress in  these direction in the near future.

\noindent{\bf Acknowledgements}

 This work was supported by CNPq (Grants  302473/2017-3 and 408151/2016-1) and FAPEG/PRONEM- 201710267000532.


\end{document}